\documentclass[conference]{IEEEtran}
\IEEEoverridecommandlockouts
\usepackage[left=48pt, right=48pt, top=57pt, bottom=46pt]{geometry}
\usepackage{cite}
\usepackage{amsmath,amssymb,amsfonts}
\usepackage{graphicx}
\usepackage{textcomp}
\usepackage{cleveref}
\usepackage{dsfont}
\usepackage{mathtools}
\usepackage{xcolor}
\usepackage{amsthm}
\usepackage{subcaption}
\usepackage{algorithm}
\usepackage{algorithmicx}
\usepackage{amsthm,amssymb,amsfonts}
\usepackage{tikz}
\usetikzlibrary{arrows.meta}
\usepackage[noend]{algpseudocode}

 \newtheorem{theorem}{Theorem}
\newtheorem{lemma}{Lemma}
\newtheorem{problem}{Problem}

\newtheorem{remark}{Remark}
\newtheorem{assumption}{Assumption}
\crefname{lemma}{Lemma}{Lemmas}
\crefname{problem}{Problem}{Problems}
\crefname{definition}{Definition}{def}
\crefname{remark}{Remark}{rem}
\crefname{section}{section}{sections}
\crefname{assumption}{Assumption}{Assumptions}
\Crefname{section}{Section}{Sections}
\Crefname{assumption}{Assumption}{Assumptions}
\def\BibTeX{{\rm B\kern-.05em{\sc i\kern-.025em b}\kern-.08em
    T\kern-.1667em\lower.7ex\hbox{E}\kern-.125emX}}
\begin{document}

\title{Chance-Constrained Covariance Steering for Discrete-Time Markov Jump Linear Systems \\
\thanks{This work was supported by the U.S. Air Force Office of Scientific
Research through research grant FA9550-23-1-0512. The authors are with the School of Aeronautics \& Astronautics, Purdue University, West Lafayette, IN 47906 USA. (email: shriva15@purdue.edu, koguri@purdue.edu)}
}
\author{\IEEEauthorblockN{Shaurya Shrivastava}
\and
\IEEEauthorblockN{Kenshiro Oguri}
}
\maketitle
\begin{abstract}
In this paper, we solve the chance-constrained covariance steering problem for discrete-time Markov Jump Linear Systems (MJLS) using a convex optimization framework. We derive the analytical expressions for the mean and covariance trajectories of time-varying discrete-time MJLS and show that they cannot be separated even without chance constraints, unlike the single-mode dynamics case. To solve the covariance steering problem, we propose a two-step convex optimization framework, which optimizes the mean and covariance subproblems sequentially. Further, we incorporate chance constraints and propose an iterative optimization framework to solve the chance-constrained covariance steering problem for MJLS. Both problems are originally nonconvex, and we derive convex relaxations which are proved to be lossless at optimality using the Karush–Kuhn–Tucker (KKT) conditions. Numerical simulations demonstrate the proposed method by achieving target covariances while respecting chance constraints under additive noise, bias, and Markovian jump dynamics.
\end{abstract}
\section{Introduction}
Covariance steering is an important problem of designing a control policy that steers a system from an initial given distribution to a target distribution while minimizing a cost function. It was first proposed in \cite{hotz_covariance_1987} as an infinite horizon covariance control problem, and has attracted significant interest recently as a finite horizon covariance steering problem\cite{chen_optimal_2016-1,chen_optimal_2016,bakolas_finite-horizon_2018, goldshtein_finite-horizon_2017}.

The covariance steering formulation was previously applied to vehicle path planning \cite{okamoto_optimal_2019} and spacecraft trajectory optimization\cite{kumagai_robust_nodate, oguri_chance-constrained_2024, ridderhof_chance-constrained_2020}. 
Such applications typically require path constraints due to safety considerations and physical limitations and have previously been incorporated in the form of probabilistic chance constraints\cite{okamoto_optimal_2018}. \cite{liu_optimal_2024} proposed a lossless convex relaxation of the chance-constrained covariance steering problem, which was later shown \cite{rapakoulias_discrete-time_nodate} to outperform existing methods \cite{bakolas_finite-horizon_2018,okamoto_optimal_2019} in computation time.

Despite recent remarkable advances in covariance steering theory for smooth dynamics, covariance steering for hybrid systems has been significantly underexplored. \cite{yu_optimal_2024} studies covariance steering for a class of hybrid systems that involve state-triggered switching. However, such approaches do not extend to tackle covariance steering for Markov jump linear systems (MJLS). MJLS is a subclass of stochastic hybrid systems, where the mode dynamics are governed by a Markov chain. They have been extensively studied in the literature with applications in fault-tolerant control \cite{li_robust_2019}, network control \cite{su_event-triggered_2021}, and macroeconomic models \cite{do_val_receding_1999}. \cite{costa_discrete-time_2005} is an excellent reference on the classical results for discrete-time MJLS.
 
Constrained control of MJLS has been an active area of research, e.g. with state and control polytopic constraints for linear time-invariant\cite{costa_constrained_1999} and periodic systems\cite{shrivastava_robust_2024}, and with second-moment constraints\cite{vargas_second_2013}. Chance constraints have previously been incorporated for MJLS\cite{flus_control_2022}, but they do not optimize the feedback control and utilize probabilistic reachable sets, which might not scale well with higher dimensions or longer time horizons. Meanwhile, the covariance steering paradigm offers a direct way to optimize for a given cost function under chance constraints and boundary conditions in a computationally tractable manner. To the authors' best knowledge, such formulations have previously not been investigated in the context of MJLS.

In this work, we focus on the optimal covariance steering problem for MJLS with chance constraints. This is motivated by the need for probabilistic motion planning approaches which face abrupt changes in dynamics, such as spacecrafts\cite{shrivastava_markov-jump_2025} and robots\cite{vargas_second_2013} with sudden actuator failure. The contributions of this work are threefold: (1) Demonstrate the coupling of MJLS mean and covariance dynamics, showing why mean-covariance separation of single-mode systems does not extend to MJLS, (2) Derive a lossless convex relaxation to solve the MJLS covariance steering problem in a computationally efficient manner, and (3) Incorporate chance constraints in the covariance steering framework and propose a sequential convex programming approach to solve the problem.
\subsubsection*{Notation}$\mathbb{R}$ and {$\mathbb{Z}_{a:b}$ denote the set of real numbers and integers} from $a$ to $b$, respectively. $I$ denotes the identity matrix of appropriate size. $\mathds{1}_{\mathcal{\omega}}$ denotes the indicator function. $\lambda_{\max}(\cdot)$, $\mathrm{det}(\cdot)$, and $\mathrm{Tr(\cdot)}$ denote the maximum eigenvalue, determinant, and trace of a matrix respectively. $\|\cdot\|$ denotes the $l_2$ norm for a vector, while $\|\cdot\|_p$ denotes its $l_p$ norm.
\section{Problem Statement}
We consider a discrete-time MJLS defined on the probability space $(\Omega,\mathcal{F},\mathbb{P})$ with a wide-sense stationary white noise $w_k\in\mathbb{R}^{n_w}$ and a mode-dependent deterministic bias $c_k(r_k)$, given by:
\begin{equation}\label{mjls}
    x_{k+1} = A_k (r_k)x_k + B_k(r_k) u_k + c_k(r_k) + G_k(r_k)w_k
\end{equation}
where $x\in\mathbb{R}^{n_x}$ and $u_k(i)\in\mathbb{R}^{n_u}$ are the state and control at time $k$ respectively. $r_k$ denotes the active mode of a homogeneous Markov chain taking values in the set $\mathcal{S}\triangleq\{1\ldots N\}$ with transition probability matrix $\mathcal{P}$ whose ($i$,$j$)th entry is $p_{ij}$. For the white noise $w_k$, we have $\mathbb{E}[w_k]=0$, $\mathbb{E}[w_kw^\top_k]=I_{n_w}$, and $\mathbb{E}[w_jw^\top_k]=0 \;\forall i,j,k, j\neq k$ by definition. To study the MJLS dynamics, we focus on the mode probability $\rho_k(i)$, and the MJLS state mean $\mu_k$ and covariance $\Sigma_k$. In MJLS literature, it has typically been easier to work with partial expectations such as $\mathbb{E}[x_k\mathds{1}_{r_k=i}]$ instead of $\mathbb{E}[x_k]$ and we define such variables for the mean and covariance below:
{\begin{align}
    &\rho_k(i) \triangleq \mathbb{E} \left[\mathds{1}_{\{r_k=i\}}\right], 
        & q_k(i) \triangleq \mathbb{E} \left[x_k \mathds{1}_{r_k=i} \right],   \notag \\
        &\overline{x}_k(i)\triangleq\mathbb{E}[x_k|r_k=i], 
        & \Sigma_k\triangleq \mathbb{E} \left[(x_k-\mu_k) (x_k-\mu_k)^\top \right]\notag\\
        &\mu_k\triangleq \mathbb{E}[x_k],& \hspace{-12mm}S_k(i) \triangleq \mathbb{E}[(x-\overline{x}_k(i))(x-\overline{x}_k(i))^\top\mathds{1}_{r_k=i}]\notag
\end{align}}
{\noindent The mode probability distribution vector is given by $\rho_k=[\rho_k(1)\ldots\rho_k(N)]$. For a given initial mode distribution $\rho_0$, $\rho_k$ can be calculated as $\rho_k=\rho_0\mathcal{P}^k$. Our framework can be extended trivially to time-inhomogeneous MJLS with $\rho_k=\rho_0\mathcal{P}_0\mathcal{P}_1\ldots\mathcal{P}_{k-1}$, but we only consider the time-homogeneous case for notational simplicity. The MJLS covariance steering optimal control problem is then defined as:
\begin{problem}\label{prob0} Solve:
    \begin{subequations}
    \begin{align}
        \min_{x_k,u_k(i)} \;J& = \sum_{k=0}^{T-1} \mathbb{E}\left[x_kQ_kx_k^\top + u_k(i)R_ku_k^\top(i)\right]\label{prob1_costfunction}&\\
        \mathrm{s.t}&.\;\;\;\forall k \in \mathbb{Z}_{0:T-1},\notag\\
        x_{k+1} = A_k& (r_k)x_k + B_k(r_k) u_k + c_k(r_k) + G_k(r_k) w_k&\\
        \mathbb{E}[x_0]=\mu_0&,\;\mathrm{cov}(x_0)={\Sigma}_0,  \;\mathbb{E}[x_T]=\mu_f, \; \mathrm{cov}(x_T)\preceq\Sigma_f&\label{bc}\\
        \mathbb{P}(x_k\in \mathcal{X}) &\geq 1 - \epsilon_x, \quad \mathbb{P}(u_k(i) \in \mathcal{U})\geq 1 - \epsilon_u\label{cc}
    \end{align}
\end{subequations}
where $Q_k\succeq0$, $R_k\succ0$ and $A_k(i)$ is invertible $\forall k\in\mathbb{Z}_{0:T},i\in\mathcal{S}$. The distribution for $\rho_0$, $w_k$ and $x_0$ are independent. ${\mu}_0,\;{\Sigma}_0$ are the initial mean and covariance, while $\mu_f,\Sigma_f$ are the final time mean and covariance constraint with $\mathrm{cov}(x)$ denoting the covariance of $x$. $\mathcal{X}$ and $\mathcal{U}$ indicate the feasible sets which must be satisfied with probability $1-\epsilon_x$ and $1-\epsilon_u$ for state and control respectively.
\end{problem}}
\begin{assumption}\label{asm_rho}
    $\rho_k(i)>0 ,\;\forall k\in{\mathbb{Z}_{0:T}},i\in\mathcal{S}$, i.e. the probability of any mode being active is non-zero at all times.
\end{assumption}
{\Cref{prob0} is hard to solve as the decision variables are random variables. To solve it in a computationally tractable manner, we focus on the mean and covariance propagation of the MJLS. However, instead of dealing with $\mu_k$ and $\Sigma_k$, the MJLS dynamics motivate us to instead work with $q_k(i),\overline{x}_k(i)$ and $S_k(i)$, which are related to $\mu_k$ and $\Sigma_k$ as:
\begin{subequations}
    \begin{align}
    \mu_k &= \sum_{i\in\mathcal{S}}q_k(i) , \;\; \overline{x}_k(i) =\frac{q_k(i)}{\rho_k(i)}\label{mean_def}\\
    \Sigma_k &= \sum_{i\in\mathcal{S}} \left(S_k(i)  + \rho_k(i)(\overline{x}_k(i)-\mu_k)(\overline{x}_k(i)-\mu_k)^\top\right) \label{cov_def}
\end{align}
\end{subequations} 
We consider a mode-dependent affine feedback control policy:}
{\begin{equation}\label{controllaw}
    {u_k(i) = \overline{u}_k(i) + K_k(i) \left(x_k-\overline{x}_k(i)\right)}
\end{equation}}
{\begin{assumption}\label{asm_cov}
    Under \cref{controllaw}, $S_k(i)\succ 0, \;\forall k \in {\mathbb{Z}_{0:T},i\in\mathcal{S}}$. 
\end{assumption}
Such control laws with a linear feedback law is common in covariance steering literature \cite{liu_optimal_2024,rapakoulias_discrete-time_nodate} as it allows us to steer both the mean and covariance simultaneously.
We will now use the first and second order moments to solve \cref{prob0} in a tractable manner. Focusing on the mean, 
$q_k(i)$ propagation can be performed as given below:}
\begin{align}
    q_{k+1}&(j)= \mathbb{E} \left[x_{k+1} \mathds{1}_{\{r_{k+1}=j\}} \right]\notag \\
        &\hspace{-5mm}=\sum_{i\in\mathcal{S}}\mathbb{E} \left[\left({A}_k (i)x_k + B_k(i) (\overline{u}_k(i) + K_k(i) (x_k-{\overline{x}_k(i)}))\right.\right.\notag\\&\left.\left. \hspace{3cm} + c_k(i) + G_k(i) w_k \right)p_{ij} \mathds{1}_{\{r_k=i\}} \right]\notag\\
        &\hspace{-5mm}={\sum_{i\in\mathcal{S}}p_{ij}\left({A}_k(i)q_k(i) +\rho_k(i)(B_k(i)\overline{u}_k(i) + c_k(i))\right)}\label{meanprop1}
\end{align}
Similarly, covariance propagation can be studied using $S_k(i)$ as follows: 
\begin{align}\label{cov1}
    S_{k+1}(j)&=\mathbb{E}[(x_{k+1}-\overline{x}_{k+1}(j))(x_{k+1}-\overline{x}_{k+1}(j))^\top\mathds{1}_{r_{k+1}=j}]\notag\\
    &\hspace{-5mm}=\mathbb{E}[x_{k+1}x_{k+1}^\top\mathds{1}_{r_{k+1}=j}] - \rho_{k+1}(j)\overline{x}_{k+1}(j)\overline{x}_{k+1}^\top(j)
\end{align}
Substituting the identities given in the appendix, we get:
\begin{align}
    & S_{k+1}(j)=\Bigl(\sum_{i\in\mathcal{S}}\left[p_{ij}\{A_k(i)S_k(i)A^\top_k(i)\right.\notag\\
    &+A_k(i)S_k(i)K_k^\top(i) B_k^\top(i) +  B_k(i)K_k(i)S_k^\top(i)A_k^\top(i) \notag\\
    &+ B_k(i)(\rho_k(i)\overline{u}_k(i)\overline{u}_k^\top(i)+ K_k(i)S_k(i)K_k^\top(i))B_k^\top(i) \notag\\
    & + \rho_k(i)(B_k(i)\overline{u}_k(i)c_k^\top(i) + B_k(i)\overline{u}_k(i)\overline{x}_k^\top A_k^\top(i)  \notag\\
    & +c_k(i)c_k^\top(i)+A_k(i)\overline{x}_k(i)\overline{u}_k^\top         (i)B_k^\top(i) \notag\\
    & +c_k(i) \overline{u}_k^\top(i)B_k^\top(i)+c_k(i)\overline{x}_k^\top(i) A_k^\top(i)+ G_k(i)G_k^\top(i) \notag\\
    & +A_k(i)\overline{x}_k(i)c_k^\top(i)+A_k(i)\overline{x}_k(i)\overline{x}_k^\top(i)A_k^\top(i))\}\Bigr) \notag\\
    &\hspace{25mm}-\rho_{k+1}(j)\overline{x}_{k+1}(j)\overline{x}_{k+1}^\top(j)\label{covprop1}
\end{align}
\begin{remark}
    A careful analysis of \cref{covprop1} along with \cref{cov_def} show that the terms involving the MJLS mean information do not cancel out for the covariance propagation, demonstrating that the mean-covariance separation of single-mode dynamics (shown in \cite{okamoto_optimal_2018}) does not extend to MJLS.
\end{remark}
For a given finite sequence of tuples $\{q_k(1)\ldots q_k(N)\}_{k\in\mathbb{Z}_{0:T}}$ and $\{\overline{u}_k(1)\ldots\overline{u}_k(N)\}_{k\in\mathbb{Z}_{0:T-1}}$ which satisfy \cref{meanprop1}, we define a corresponding mean trajectory $\tau$
as a set of two finite sequences of $N+1$ and $N$-tuples $\{\mu_k,\overline{x}_k(1)\ldots\overline{x}_k(N)\}_{k\in\mathbb{Z}_{0:{T}}}$  $\{\overline{u}_k(1)\ldots\overline{u}_k(N)\}_{k\in\mathbb{Z}_{0:{T-1}}}$ respectively which satisfy \cref{mean_def} and \cref{meanprop1}  simultaneously. 

We now redefine the cost function $J$ in terms of the newly defined mean and covariance variables. From \cref{prob1_costfunction}, we know $J=\sum_{k=0}^{T-1}J_k$, where $J_k$ is given by
\begin{equation*}
    \begin{aligned}
        J_k&= \sum_{i\in\mathcal{S}}\mathbb{E}\left[(x_k^\top Q_kx_k + u^\top_k(i)R_ku_k(i))\mathds{1}_{r_k=i}\right]\\
        &=\sum_{i\in\mathcal{S}} \mathrm{Tr}\left(\mathbb{E}[x_kx_k^\top\mathds{1}_{r_k=i}]Q_k + \mathbb{E}[u_k(i)u^\top_k(i)\mathds{1}_{r_k=i}] R_k  \right)
    \end{aligned}
\end{equation*}
Using the identities derived in \cref{id0,id3} in the appendix, rewrite $J_k$ as:
\begin{align}\label{Jk}
    J_k&=\sum_{i\in\mathcal{S}} \left[\rho_k(i)\{\mathrm{Tr}(\overline{x}_k(i)\overline{x}_k^\top(i) Q_k) +\mathrm{Tr}(\overline{u}_k(i)\overline{u}^\top_k(i)R_k)\}  \right.\notag\\&\left.\hspace{0.5cm}+ \mathrm{Tr}({S}_k(i)Q_k) + \mathrm{Tr}(K_k(i)S_k(i)K^\top_k(i)R_k) \right]
\end{align}
We now define our optimal control problem in terms of the mean and covariance variables using \cref{controllaw}, which approximates \cref{prob0}. 
\begin{problem}\label{prob0_det}
    Solve:
    \begin{align}
            &\min_{\mu_k,q_k(i),\overline{x}_k(i),\overline{u}_k(i),K_k(i),S_k(i),\Sigma_k}\;J = \sum_{k=0}^{T-1} J_k\\
        \quad &\mathrm{s.t.}\;\forall k \in \mathbb{Z}_{0:T-1},\notag\\
        &\mathrm{\cref{mean_def}, \cref{meanprop1}, \cref{cov_def}, \cref{covprop1},\cref{bc},\cref{cc}}\notag
        \end{align}
\end{problem}
\cref{prob0_det} is still not a tractable problem to solve because of the nonlinear constraint \cref{covprop1} and chance constraints \cref{cc}, which are not yet defined as a deterministic constraint in the mean and covariance variables. We tackle these challenges in the following sections by first proposing a lossless convex reformulation of \cref{prob0_det} without chance constraints, and then showing that these relaxations remain lossless with the deterministic formulations of chance constraints. Our lossless relaxations convert an equality constraint to an inequality constraint, which is then shown to be tight, i.e. equivalent to the original equality constraint at optimality.
\section{Unconstrained MJLS Covariance Steering}
\subsection{Optimal Control Problem}
The unconstrained optimal control problem is given by:
\begin{problem}\label{ocp_uc}
    Solve \cref{prob0_det} without constraints in \cref{cc}.
\end{problem}
\subsection{Convex Relaxation}
The nonlinear terms in \cref{covprop1} result in a non-convex optimization landscape for \cref{ocp_uc}. We first relax the nonlinear terms involving the covariance variables $K_k(i)$ and $S_k(i)$ in \cref{covprop1} by introducing the following change of variables:
\begin{equation}\label{substitution}
    \begin{aligned}
        L_k(i)&=K_k(i){{S}_k(i)},\; & Y_k(i)&=L_k(i){{S}^{-1}_k(i)}L_k^\top(i)
    \end{aligned}
\end{equation}
Rearranging \cref{covprop1} using the above-mentioned terms,
{\begin{align}\label{covprop_LY}
    & S_{k+1}(j)=\Bigl(\sum_{i\in\mathcal{S}}\left[p_{ij}\{A_k(i)S_k(i)A^\top_k(i)\right.\notag\\
    &+A_k(i)L_k^\top(i) B_k^\top(i) +  B_k(i)L_k(i)A_k^\top(i) + B_k(i)Y_k(i)B_k^\top(i)\notag\\
    &+\rho_k(i)B_k(i)\overline{u}_k(i)(A_k(i)\overline{x}_k(i)+B_k(i)\overline{u}_k(i)+c_k(i))^\top\notag\\
    & +\rho_k(i)c_k(i)(A_k(i)\overline{x}_k(i)+B_k(i)\overline{u}_k(i)+c_k(i))^\top \notag\\
    &+\rho_k(i)A_k(i)\overline{x}_k(i)(A_k(i)\overline{x}_k(i)+B_k(i)\overline{u}_k(i)+c_k(i))^\top \notag\\
    & + G_k(i)G_k^\top(i))\}\Bigr)-\rho_{k+1}(j)\overline{x}_{k+1}(j)\overline{x}_{k+1}^\top(j)
\end{align}}
For a given mean trajectory $\tau$, we define covariance trajectory $\Xi$ as a set of two finite sequences of $N+1$ and $2N$-tuples $\{\Sigma_k\}_{k\in\mathbb{Z}_{0:T}}$ and $\{L_k(1)\ldots L_k(N),Y_k(1)\ldots Y_k(N)\}_{k\in\mathbb{Z}_{0:T-1}}$ which satisfy \cref{covprop_LY}. Thus, \cref{covprop_LY} can be expressed as
\begin{equation}\label{covprop_sub}
    H_k^\Sigma = S_{k+1}(j) - F_\tau(\Xi)=0
\end{equation}
where $F_\tau(\Xi)$ is a linear operator in $\Xi$ for a given $\tau$. We relax the non-convex definition of $Y_k(i)$ in \cref{substitution} as:
\begin{equation}\label{constraint_Ck}
    C_k(i)=L_k(i){{S}^{-1}_k(i)}L_k^\top(i)-Y_k(i)\preceq 0
\end{equation}
 which can be recast as a LMI using Schur's complement:
\begin{equation}\label{rel1}
        \begin{bsmallmatrix}
            Y_k(i) & L_k(i)\\
            L_k^\top(i) & {{S}_k(i)}
        \end{bsmallmatrix} \succeq 0
\end{equation}
We note that \cref{covprop_LY} still has nonlinear terms involving the mean variables. To solve \cref{ocp_uc} in a computationally tractable manner, we propose a two-step optimization scheme which solves the mean and covariance problem sequentially. The objective function can be divided into a mean and covariance part using the new variables as follows:
\begin{equation}\label{Jk_sub}
    \begin{aligned}
        J&=J(\tau)+J(\Xi)
    \end{aligned}
\end{equation}
where 
{\begin{subequations}
    \begin{align}
        J(\tau)=\sum_{k=0}^{T-1}&\sum_{i\in\mathcal{S}}\rho_k(i)\{\mathrm{Tr}(\overline{x}_k(i)\overline{x}_k^\top(i) Q_k +\overline{u}_k(i)\overline{u}^\top_k(i)R_k)\}\label{Jtau}\\
        J(\Xi) &=\sum_{k=0}^{T-1}\sum_{i\in\mathcal{S}} \left[\text{Tr}({S}_k(i)Q_k+Y_k(i)R_k)\right]\label{JXi}
    \end{align}
\end{subequations}}
\subsection{Two-step Convex Optimization Framework}
We now define the two subproblems which comprise our two-step optimization framework. It first solves \cref{prob1} and then uses its solution to solve \cref{prob2}.

\begin{problem}\label{prob1} {Solve:
\begin{subequations}\label{ocp_mean}
\begin{align}
    &\min_{q_k(i),\overline{x}_k(i),\overline{u}_k(i)} J(\tau)& \label{eq:mean_obj} \\
    &\mathrm{s.t.}\quad\quad\mathrm{\cref{meanprop1}, \cref{mean_def}}\;\;\;\forall k \in \mathbb{Z}_0^{T-1}, \;\;\forall i \in \mathcal{S} \label{meanprob_uc1}
    &\\
    &q_0(i)=\rho_0(i){\mu}_0 , \;\sum_{i\in\mathcal{S}}q_T(i)=\mu_F\label{meanprob_uc2}
\end{align}
\end{subequations}}
\end{problem}

\begin{problem} \label{prob2}
    Given a mean trajectory $\tau$, solve:
{\begin{subequations}\label{ocp_cov}
\begin{align}
    &\hspace{-3cm}\min_{S_k(i),L_k(i),Y_k(i)} J(\Xi)\label{eq:cov_obj} \\
    \mathrm{s.t.}\;\;\; \mathrm{\cref{covprop_sub}, \cref{rel1}}\;\; &\forall k \in \mathbb{Z}_0^{T-1}, \; i \in \mathcal{S}, \label{covprob_uc1}\\
    &\hspace{-3cm}S_0(i)=\rho_0(i){\Sigma}_0, \quad \Sigma_T \preceq \Sigma_f \label{covprob_uc2}
\end{align}
\end{subequations}
where $\Sigma_T$ is calculated with $S_T(i)$ using \cref{cov_def}}
\end{problem}

\begin{assumption}\label{asm_slater}
    Slater's condition is satisfied for \cref{prob2}, i.e. there exists a strictly feasible solution for \cref{prob2}.
 \end{assumption}

\begin{lemma}[Lemma 1 in \cite{liu_optimal_2024}]\label{lemma1}
    Let $A$ and $B$ be $n \times n$ symmetric matrices with $A \succeq 0, B \preceq 0$, and $\mathrm{Tr}(A B)=0$. If $B$ has at least one nonzero eigenvalue, then, $A$ is singular.
\end{lemma} 

\begin{theorem}\label{thm1}
    Under \cref{asm_rho,asm_cov,asm_slater}, the solution to \cref{prob2} satisfies $Y_k(i)=L_k(i){{S}^{-1}_k(i)}L_k^\top(i)$, hence, lossless.
\end{theorem}
\begin{proof}
For \cref{prob2}, we can define the Lagrangian as:
\begin{equation}
\begin{aligned}
        \mathcal{L}(\cdot) = J(\Xi) +&\sum_{k=0}^{T-1}\left[ \mathrm{Tr}(\Lambda^\top_kH_k^\Sigma) +\sum_{i\in\mathcal{S}}\left\{\mathrm{Tr}(M^\top_k(i)C_k(i)) \right\}\right] \\
        & + \mathrm{Tr}(\lambda^\top_F(\Sigma_T-\Sigma_f))
\end{aligned}
\end{equation}
where $C_k(i)$ is defined in \cref{constraint_Ck}, $\Lambda_k$, $M_k$ and $\lambda_F$ are Lagrange multipliers corresponding to $H_k^\Sigma=0$, $C_k(i)\preceq0$ and $\Sigma_T\preceq\Sigma_f$ respectively. $\Lambda_k$ is symmetric because of symmetry in \cref{covprop_sub} and $M_k$ and $\lambda_F$ are symmetric by definition. The relevant first-order KKT conditions for optimality are:
\begin{subequations}
    \begin{align}
        \tfrac{\partial \mathcal{L}}{\partial L_k(i)} &= 2\sum_{j\in\mathcal{S}}\pi^{(k)}_{ij}B_k^\top(i)\Lambda_kA_k(i)  +2M_k(i)L_k(i){S}_k^{-1}(i)=0 \label{lm1} \\
        \tfrac{\partial \mathcal{L}}{\partial Y_k(i)} &= \rho_k(i)R_k-M_k(i) +\sum_{j\in\mathcal{S}}\pi^{(k)}_{ij}B^\top_k(i)\Lambda_kB_k(i)=0 \label{lm2}\\
        \mathrm{Tr}(&M_k^\top(i)C_k(i))=0, \quad C_k\preceq 0, \quad M_k\succeq 0\label{cs1}
    \end{align}
\end{subequations}
where $\pi^{(k)}_{ij}=p_{ij}\rho_k(i)$. We now prove by contradiction that $C_k(i)=0$, or the relaxation is lossless. Let us suppose that $C_k(i)$ has at least one nonzero eigenvalue, which implies the corresponding Lagrange multiplier $M_k(i)$ is singular from \cref{lemma1} and \cref{cs1}. 
Rearranging \cref{lm1}, we get $\sum_{j\in\mathcal{S}}\pi^{(k)}_{ij}B_k^\top(i)\Lambda_k=-M_k(i)L_k(i){S}_k^{-1}(i)A_k^{-1}(i)$. Using this and the substitutions defined earlier in \cref{substitution}, we simplify \cref{lm2} as: $\rho_k(i)R_k=M_k(i)(I +K_k(i)A_k^{-1}(i)B_k^\top(i))$. Taking the determinant of both sides, we get the following:
\begin{equation*}\label{thm1_p}
    \mathrm{det}(\rho_k(i)R_k)=\mathrm{det}(M_k(i))\mathrm{det}(I +K_k(i)A_k^{-1}(i)B_k^\top(i))=0
\end{equation*}
which contradicts $\rho_k(i)R_k\succ 0$ from \cref{asm_rho}. Therefore, we have $C_k(i)=L_k(i)S^{-1}_k(i)L_k^\top(i)-Y_k(i)=0,\; \forall k,i.$
\end{proof}
\begin{remark}
{Similar relaxations cannot be performed for mean quadratic terms in the \cref{covprop1} because they involve rank-1 matrices which doesn't allow the use of Schur's complement to relax the non-convex substitutions as LMIs, motivating our two-step sequential mean-covariance optimization approach.}
\end{remark}
\section{Chance-Constrained MJLS Covariance Steering}
We now consider chance constraints for both state and control. For state, we consider intersection of hyperplane constraints and tube-like constraints around the mean trajectory. Mathematically, they're expressed as:
\begin{subequations}\label{cc_x}
    \begin{align}
        \mathbb{P}(a^\top_j x_k + b_j\leq0, \forall j) &\geq 1-\delta_{x}\label{cc_state}\\
    \mathbb{P}(\|x_k-\mu_k\|\leq d_{\max}) &\geq 1 - \epsilon_x\label{cc_unorm}
    \end{align}
\end{subequations}
Likewise, for control, we consider
\begin{subequations}\label{cc_u}
    \begin{align}
        \mathbb{P}(f^\top_j u_k(i) + g_j \leq0, \forall j) &\geq 1-\delta_{u}(i)\label{cc_state}\\
    \mathbb{P}(\|u_k(i)\|\leq u_{\max}(i)) &\geq 1 - \epsilon_u(i)\label{cc_unorm}
    \end{align}
\end{subequations}

Previous works \cite{ono_probabilistic_2013,oguri_chance-constrained_2024,okamoto_optimal_2018,okamoto_optimal_2019} have focused on deterministic and convex reformulations of \cref{cc_x,cc_u}, but they usually assume state distribution to be Gaussian. MJLS dynamics result in a mixture of exponential Gaussians ($N^T$ Gaussians for $N$ modal MJLS at time $T$.) Therefore, we utilize one-sided Chebyshev/Cantelli's inequality\cite{durrett_probability_2019} and the multivariate Chebyshev inequality\cite{chen_new_2011} to convert the chance constraints to deterministic path constraints. Note that while the constraint on the state $x_k$ is mode-independent, the control chance constraints are imposed for a given mode $i$. The multi-modal nature of the control variable $u_k$ and practical applications motivate us to consider mode-dependent control chance constraints as it would be less conservative than mode-independent control chance constraints. For mode-conditioned control chance constraints, the mean and covariance is given by: $\mathbb{E}[u_k(i)|r_k=i]=\overline{u}_k(i)$ and $\mathbb{E}[(u_k(i)-\overline{u}_k(i))(u_k(i)-\overline{u}_k(i))^\top|r_k=i]=\frac{Y_k(i)}{\rho_k(i)}$ as shown in the appendix.

\subsection{Deterministic Formulation}
\Cref{det_cc_halfspace,det_cc_norm} are useful in the deterministic formulation of half-space and max-norm bound constraints.
\begin{lemma}\label{det_cc_halfspace}
    Let $v \in \mathbb{R}^{n_v}$ be a random variable with mean $\overline{v}$ and covariance $V$. Then, $\mathbb{P}[a_j^\top {v} + b_j \leq 0,\forall j]\geq 1-\epsilon$ is implied by: 
    \begin{align}\label{det_cc_halfspace_eq}
         a^\top_j\overline{v} + b_j + \sqrt{\frac{1-\epsilon_j}{\epsilon_j}  a^\top_j V a_j} \leq 0, \quad\sum_j\epsilon_j\leq \epsilon
    \end{align}
\end{lemma}
\begin{proof}
    See Theorem 3.1 in \cite{calafiore_distributionally_2006}, and Lemma 1 in \cite{oguri_chance-constrained_2024}.
\end{proof}

\begin{lemma}\label{det_cc_norm}
    Let $v \in \mathbb{R}^{n_v}$ be a random variable with mean $\overline{v}$ and covariance $V$. Then  $\mathbb{P}[\|{v}\| \leq v_{\max}] \geq 1 - \epsilon$ is implied by:
    \begin{align}\label{det_cc_norm_eq}
        \|\overline{v}\| +\sqrt{\frac{n_{v}}{\epsilon} \cdot\lambda_{\max}(V)} \leq v_{\max}
    \end{align}
\end{lemma}
\begin{proof}

    Given that $\|\overline{v}\| +\sqrt{\frac{n_{v}}{\epsilon} \cdot\lambda_{\max}(V)} \leq v_{\max}$, we know $\mathbb{P}(\|v\|\leq v_{\max})\geq\mathbb{P}(\|v\|\leq \|\overline{v}\| +\sqrt{\frac{n_{v}}{\epsilon} \cdot\lambda_{\max}(V)})$\ldots (A). This is described in \Cref{fig:lemma_schematic_b}. Now using the multivariate Chebyshev inequality \cite{chen_new_2011} for $v$, we get: $
    \mathbb{P}((v-\overline{v})^\top  V^{-1} (v-\overline{v})\geq t^2) \leq \frac{n_v}{t^2}$\ldots (B). Put $t=\sqrt{\frac{n_v}{\epsilon}}$, and considering the complement of the event in (B), we get $\mathbb{P}((v-\overline{v})^\top  V^{-1} (v-\overline{v})\leq \frac{n_v}{\epsilon})\geq 1-\epsilon$. Note that $(v-\overline{v})^\top  V^{-1} (v-\overline{v})\leq \frac{n_v}{\epsilon}$ denotes an ellipsoid centered at $\overline{v}$ with shape matrix $V$. We first find the maximum value of $\|v\|$ inside this ellipse. Let $\mathcal{E}=\{v:(v-\overline{v})^\top  V^{-1} (v-\overline{v})\leq \frac{n_v}{\epsilon}\}$, then: $\max_{v\in\mathcal{E}} \|v\| \leq \max_{v\in\mathcal{E}}\|v-\overline{v}\| + \|\overline{v}\|=\sqrt{\frac{n_v}{\epsilon}\lambda_{\max}(V)} + \|\overline{v}\|$. Now construct a ball $\mathcal{B}_v$ centered at origin with radius $\|\overline{v}\| +\sqrt{\frac{n_{v}}{\epsilon} \cdot\lambda_{\max}(V)}$. Geometrically, this is shown in \Cref{fig:lemma_schematic_a} for a $n_v=2$. Then, $\mathbb{P}(v\in\mathcal{B}_v)=\mathbb{P}(\|v\|\leq \sqrt{\frac{n_v}{\epsilon}\lambda_{\max}(V)} + \|\overline{v}\|)\geq \mathbb{P}(v\in\mathcal{E})$ since $\mathcal{E}\subseteq\mathcal{B}_v$. But $\mathbb{P}(v\in\mathcal{E})\geq 1-\epsilon$. Therefore, $\mathbb{P}\left(\|v\|\leq \|\overline{v}\| +\sqrt{\frac{n_{v}}{\epsilon}\lambda_{\max}(V)}\right)\geq\mathbb{P}(v\in\mathcal{E})\geq 1-\epsilon$. Combining this with (A), we get $\mathbb{P}(\|v\|\leq v_{\max})\geq\mathbb{P}(\|v\|\leq \|\overline{v}\| +\sqrt{\frac{n_{v}}{\epsilon} \cdot\lambda_{\max}(V)})\geq 1-\epsilon$.
 \begin{figure}[ht]
    \centering 
    \begin{subfigure}[b]{0.2\textwidth}
        \centering
        
        \begin{tikzpicture}[scale=0.5]

            \def\muval{1.5}
            \def\vmaxval{5.5} 
            
            \def\curvepath{(-0.5, 0.01) .. controls (0.4, 4.0) and (0.8, 4.0) .. (1.5, 2.5) .. controls (2.2, 1.0) and (4.5, 0.3) .. (7.5, 0.01)} 
            
            
            \begin{scope}
                \clip (-0.5, 0) rectangle (\vmaxval, 4.2); 
                \fill[blue, opacity=0.15] \curvepath -- (\vmaxval, 0) -- (-0.5, 0) -- cycle;
            \end{scope}
        
            \begin{scope}
                \clip (-0.5, 0) rectangle (\muval, 4.2);
                \fill[red, opacity=0.2] \curvepath -- (\muval, 0) -- (-0.5, 0) -- cycle;
            \end{scope}
        
            \draw[thick] \curvepath;
        
            \draw[-{Stealth[length=2mm]}] (-1, 4.5) -- (8, 4.5); 
            
            \draw (\muval, 4.6) -- (\muval, 4.4) node[above, font=\small, red] {$||\overline{v}||+\sqrt{\frac{n}{\epsilon}\lambda_{max}(V)} $};
            \draw (\vmaxval, 4.6) -- (\vmaxval, 4.4) node[above, font=\small, blue] {$v_{\max}$};
            
            \draw[-{Stealth[length=2mm]}] (-1, 0) -- (8, 0) node[below left] {$||v||$}; 
            \draw (0, -0.2) -- (0, 4.2);
            
            \node at (-0.6, 3.75) {${p(\cdot)}$};
            
            \draw[dashed] (\muval, 4.5) -- (\muval, 0);
            \draw[dashed] (\vmaxval, 4.5) -- (\vmaxval, 0);
            
        \end{tikzpicture}
        \caption{}
        \label{fig:lemma_schematic_b}
    \end{subfigure}
    \hspace{8mm}
    \begin{subfigure}[b]{0.2\textwidth}
        \centering
        \begin{tikzpicture}[scale=0.3]
            \draw[-{Stealth[length=3mm]}] (-5,0) -- (5,0) node[below right] {$v_1$};
            \draw[-{Stealth[length=3mm]}] (0,-5) -- (0,5) node[above left] {$v_2$};
            \node at (0,0) [below left] {$0$};

            \coordinate (xi_bar_center) at (1.8, 1.4);

            \draw[red, fill=red!20, opacity=0.7] (0,0) circle (3.88);

            \node[red, align=left] at (3.2, 3.8) {$\mathcal{B}_v$};

            \begin{scope}[shift={(xi_bar_center)}, rotate=35]
                \draw[blue, fill=blue!30, opacity=0.8] (0,0) ellipse (1.6 and 0.9);
            \end{scope}

            \node[blue, align=center] at (2.8, 0.5) {$\mathcal{E}$};

            \draw[-{Stealth[length=2mm]}, thick] (0,0) -- (xi_bar_center) node[midway, above left, black] {$\overline{v}$};
            \fill (xi_bar_center) circle (1.5pt);
        \end{tikzpicture}
        \caption{}
        \label{fig:lemma_schematic_a}
    \end{subfigure}%
    

    \caption{Schematics for the proof of \cref{det_cc_norm}. Figure (a) shows that $\mathbb{P}(\|v\|\leq v_{\max})\geq \mathbb{P}\left(\|v\|\leq\|\overline{v}\| + \sqrt{\frac{n}{\epsilon}\lambda_{\max}(V)}\right)$. Figure (b) shows that $\mathbb{P}(v\in\mathcal{B}_v)\geq\mathbb{P}(v\in\mathcal{E})$ where the ball $\mathcal{B}_v$ is given by $\{v:\|v\|\leq\|\overline{v}\|+\sqrt{\frac{n_v}{\epsilon}\lambda_{\max}(V)}\}$ and the ellipsoid $\mathcal{E}$ is given by $\{v:(v-\overline{v})^\top  V^{-1} (v-\overline{v})\leq \frac{n_v}{\epsilon}\}$.}
    \label{fig:both_diagrams}
\end{figure}
\end{proof}
\begin{remark}
    \noindent{Note that while \cref{det_cc_halfspace_eq,det_cc_norm_eq} are linear in the mean variables, they are nonconvex in the covariance terms. We convexify them by rearranging and squaring both sides as:}
\begin{equation}\label{cc_x_square}
    \begin{aligned}
         \frac{1-\epsilon}{\epsilon}  a^\top_j V a_j \leq  (a_j^{\top}\bar{v} + b_j)^2
    \end{aligned}
\end{equation}
\begin{equation}\label{cc_u_square}
    \begin{aligned}
         \frac{n_{v}}{\epsilon}{\lambda_{\max}(V)} \leq (v_{\max}-\|\bar{v}\|)^2
    \end{aligned}
\end{equation}
which are convex in the covariance $V$.
\end{remark}

\subsection{Iterative Convex Control Synthesis Framework}

Using the deterministic formulation of chance constraints in \cref{det_cc_halfspace,det_cc_norm}, we can now incorporate them in our two-step optimization framework. Note that the chance constraints introduce two-way coupling in the problem, while previously only the covariance problem depended on the mean trajectory and not vice-versa. To solve this, we propose a hierarchical optimization framework which iteratively solves the mean and covariance optimization problem by relaxing the chance constraints. We relax the constraints by adding a slack variable, and then penalize its 1-norm in the objective function. Using the 1-norm promotes sparsity of the slack variable, and has demonstrated favorable convergence properties in other hierarchical optimization frameworks \cite{mao_successive_nodate,oguri_successive_2023}. We now define the relaxed mean and covariance problems in \cref{prob3} and \cref{prob4} respectively as follows:
\begin{problem}\label{prob3}
    Given a covariance trajectory $\Xi$, solve:
    \begin{subequations}\label{ocp_mean_rel}
    \begin{align}
    &\min_{\mu_k,\overline{u}_k(i)} J_\mu^c& \label{eq:mean_obj} \\
    &\mathrm{s.t.}\;\;\;\forall k \in \mathbb{Z}_0^{T-1}, i \in\mathcal{S},j \nonumber \\  
    &\quad\quad\mathrm{\cref{meanprob_uc1}, \cref{meanprob_uc2}}&\\
    a_j^\top \mu_k& + b_j + \sqrt{\frac{1-\delta_{x,j}}{\delta_{x,j}}a_j^\top \Sigma_k a_j} \leq \beta_{1,k}, \\
        f_j^\top \overline{u}_k(i)& + g_j + \sqrt{\frac{1-\delta_{u,j}(i)}{\delta_{u,j}(i)}p_j^\top \frac{Y_k(i)}{\rho_k(i)}p_j} \leq \beta_{2,k}\\
     \|\overline{u}_k(i)\|+& \sqrt{\frac{n_u(i)}{\epsilon_u(i)}\lambda_{\max}(\frac{Y_k(i)}{\rho_k(i)})} \leq u_{\max}(i) + \beta_{3,k}
    \end{align}
    \end{subequations}
where $\sum_j\delta_{x,j}\leq\delta_x$, $\sum_j\delta_{u,j}(i)\leq\delta_u(i)$, $J_\mu^c=J(\tau) + \sum_m \alpha_m \|\beta_m\|_1$, with $J(\tau)$ defined in \cref{Jtau}. $\beta_{m,k}>0$ is the slack variable for $m$-th constraint at time step $k$ and $\beta_m=[\beta_{m,0}\ldots\beta_{m,T-1}]^\top$. $\alpha_m>0$ is the weight corresponding to the slack variable vector 1-norm $\|\beta_m\|_1$ in $J_\mu^c$.
\end{problem}

\begin{problem} \label{prob4}
    Given a mean trajectory $\tau$, solve:
    \begin{subequations}\label{ocp_cov_cc}
\begin{align}
    \min_{{\Sigma}_k,L_k(i),Y_k(i)} J^c_\Sigma\hspace{2cm}& \label{eq:cov_obj} \\
    \mathrm{s.t.}\;\;\; \forall k \in \mathbb{Z}_0^{T-1}, \; i \in \mathcal{S}, j, \nonumber \hspace{5mm}&\\
    \mathrm{\cref{cov_def}, \cref{covprob_uc1},\cref{covprob_uc2}}&\\
         {\frac{1-\delta_{x,j}}{\delta_{x,j}}a_j^\top \Sigma_k a_j} - (a_j^\top \mu_k + b_j)^2& \leq\zeta_{1,j}, \label{c1} \\
        {\frac{1-\delta_{u,j}(i)}{\delta_{u,j}(i)}p_j^\top\frac{Y_k(i)}{\rho_k(i)} p_j} - (f_j^\top \mu_k + g_j)^2& \leq\zeta_{2,j} , \label{c2} \\
        \frac{n_x}{\epsilon_x}\lambda_{\max}(\Sigma_k)&\leq\zeta_{3,j}\label{c3} \\
        \frac{n_u(i)}{\epsilon_u(i)}\lambda_{\max}(\frac{Y_k(i)}{\rho_k(i)}) - u_m^2&\leq\zeta_{4,j} \label{c4}
\end{align}
\end{subequations}
where $\sum_j\delta_{x,j}\leq\delta_x$, $\sum_j\delta_{u,j}(i)\leq\delta_u(i)$, $u_m=u_{\max} - \|\overline{u}_k(i)\|$, and $ J_\Sigma^c=J(\Xi)+\sum_j \alpha_m\|\zeta_m\|_1$ with $J(\Xi)$ defined in \cref{JXi}. $\zeta_{m,k}>0$ is the slack variable for $m$-th constraint at time $k$ and $\zeta_m=[\zeta_{m,0}\ldots\zeta_{m,T-1}]^\top$. $\alpha_m>0$ is the weight for slack variable vector 1-norm $\|\zeta_m\|_1$ in $J_\Sigma^c$.
\end{problem}
\begin{assumption}\label{asm_slater2}
    Slater's condition is satisfied in \cref{prob4} for $\zeta_m=0$. This ensures the convex relaxation in \cref{rel1} remains lossless while satisfying chance constraints at $\zeta_m=0$.
\end{assumption}
We now show that the convex relaxation $Y_k(i)\succeq L_k(i)\Sigma_k^{-1}L_k^\top(i)$ remains lossless in \cref{prob4}. 
\begin{theorem}
    Under \cref{asm_rho,asm_cov,asm_slater2}, the optimal solution to \cref{prob4} satisfies $Y_k(i)=L_k(i){{S}^{-1}_k(i)}L_k^\top(i)$.
\end{theorem}
\begin{proof}
We can define the Lagrangian for \cref{prob4} as follows:
\begin{equation}
\begin{aligned}
        \mathcal{L}(\cdot) =& J_\Sigma^c +\sum_{k=0}^{T-1}\Bigl[ \mathrm{Tr}(\Lambda^\top_kH_k^\Sigma)+ \lambda_{x,n} e_{x,n}+ \lambda_{x,h} e_{x,h}\\&+\sum_{i\in\mathcal{S}}\left\{\mathrm{Tr}(M^\top_k(i)C_k(i)) + \lambda_{u,n}(i) e_{u,n}(i) \right.\\
        &\left.+ \lambda_{u,h}(i) e_{u,h}(i) \right\}\Bigr] + \mathrm{Tr}(\lambda^\top_F(\Sigma_T-\Sigma_f))
\end{aligned}
\end{equation}
where $C_k(i)$ is given in \cref{constraint_Ck} and $e_{x,h}$, $e_{u,h}$, $e_{x,n}$, and $e_{u,n}$ represent \cref{c1,c2,c3,c4} respectively with $\lambda_{x,h}$, $\lambda_{u,h}$, $\lambda_{x,n}$, and $\lambda_{u,n}$ being the respective Lagrange multipliers. $\Lambda_k$,$M_k$ and $\lambda_F$ are Lagrange multipliers corresponding to $H_k^\Sigma=0$, $C_k(i)\preceq0$ and $\Sigma_T\preceq\Sigma_f$ respectively, and are all symmetric as explained in the proof of \cref{thm1}. The relevant first-order KKT conditions for optimality are:
\begin{subequations}
    \begin{align}
        &\tfrac{\partial \mathcal{L}}{\partial L_k(i)} = 2\sum_{j\in\mathcal{S}}\pi^{(k)}_{ij}B_k^\top(i)\Lambda_kA_k(i)+2M_k(i)L_k(i)S_k^{-1}(i)\notag\\&\hspace{7cm}=0 \label{lm1c} \\
        &\tfrac{\partial \mathcal{L}}{\partial Y_k(i)} = \rho_k(i)R_k-M_k(i)  + \lambda_{u,n}\tfrac{n_u(i)}{\epsilon_u(i)} v_k(i)v_k^\top(i) \nonumber \\
        &+\lambda_{u,h} \tfrac{1-\delta_{u,j}(i)}{\delta_{u,j}(i)} f_j f^\top_j+\sum_{j\in\mathcal{S}}\pi^{(k)}_{ij}B^\top_k(i)\Lambda_kB_k(i)=0 \label{lm2c}\\
        &\mathrm{Tr}(M_k^\top(i)C_k(i))=0, \quad C_k\preceq 0, \quad M_k \succeq 0,\nonumber\\
        & \hspace{2cm}\lambda_{x,n},\lambda_{x,h},\lambda_{u,h},\lambda_{u,n}\geq0\label{cs2}
    \end{align}
\end{subequations}
\noindent{where $\pi^{(k)}_{ij}=p_{ij}\rho_k(i)$ and $v_k(i)$ is the eigenvector of $Y_k(i)$ corresponding to its maximum eigenvalue. We now prove by contradiction that $C_k(i)=0$, or the relaxation is lossless. Let us suppose that $C_k(i)$ has at least one nonzero eigenvalue, which implies $M_k(i)$ is singular from \cref{lemma1} and \cref{cs2}. Rearranging \cref{lm1c}, we get $\sum_{j\in\mathcal{S}}\pi^{(k)}_{ij}B_k^\top(i)\Lambda_k=-M_k(i)L_k(i)S_k^{-1}(i)A_k^{-1}(i)$. Using this and substitutions defined in \cref{substitution}, we simplify \cref{lm2c} as: $\rho_k(i)R_k+\lambda_{u,n}\tfrac{n_u(i)}{\epsilon_u(i)}  v_k(i)v_k^\top(i)+\lambda_{u,h} \tfrac{1-\delta_{u,j}(i)}{\delta_{u,j}(i)} f_j f^\top_j=M_k(i)(I +K_k(i)A_k^{-1}(i)B_k^\top(i))$. Representing the LHS as $\Delta_k(i)$ and taking the determinant of both sides, we get:}
\begin{equation*}\label{t2p1}
    \mathrm{det}(\Delta_k(i))=\mathrm{det}(M_k(i))\mathrm{det}(I +K_k(i)A_k^{-1}(i)B_k^\top(i))=0
\end{equation*}
From \cref{asm_rho}, we have that $\rho_k(i)R_k\succ 0$. Now we know $\lambda_{u,n}\tfrac{n_u(i)}{\epsilon_u(i)}  v_k(i)v_k^\top(i)+\lambda_{u,h} \tfrac{1-\delta_{u,j}(i)}{\delta_{u,j}(i)} f_j f^\top_j\geq0$ which implies $\Delta_k(i)>0$, resulting in a contradiction. Hence, $C_k(i)=L_k(i){S}^{-1}_k(i)L_k^\top(i)-Y_k(i)=0\;\forall k\in\mathbb{Z}_0^{T-1}, i \in\mathcal{S}$.
\end{proof}
\subsection{Sequential Convex Optimization Framework}
We now describe the sequential convex optimization framework which is used to design optimal control laws for MJLS covariance steering with chance constraints. The framework requires a mean trajectory to be initialized (which is obtained by optimizing the unconstrained mean trajectory in \cref{prob1}) and then solve \cref{prob4} and \cref{prob3} iteratively with increasing weights on the slack variables until they reach negligible values. This framework is described in \Cref{algorithm_1}.

\begin{algorithm}[H]
    \caption{MJLS Chance-constrained covariance steering}\label{algorithm_1}
    \begin{algorithmic}[1]
        \State Set tol, initial weights $\{\alpha_m\}$, and $\eta>1$;
        \State Initialize mean trajectory $\tau$ by solving \cref{prob1};
        \While{NotConverged}
        \State Solve \cref{prob4} with last stored mean trajectory $\tau$ to get new covariance trajectory $\Xi$;
        \State Compute new mean trajectory $\tau$ by solving \cref{prob3} using revised covariance trajectory $\Xi$ from previous step ;
        \State Store values of slack variables vector $\beta_m$ and $\zeta_m$;
        \If{$\max_m\|\beta_m\|_\infty$$>$tol \textbf{or} $\max_m\|\zeta_m\|_\infty$ $>$tol}
        \State increase weights $\{\alpha_m\}$ by a factor $\eta$;
        \Else
        \State Solution converged. \textbf{break;}
        \EndIf
\EndWhile
    \State \Return $\tau$ and $\Xi$
    \end{algorithmic}
\end{algorithm}

\section{Numerical Simulations}

In this section, we demonstrate our framework in two different numerical examples. For the first example, we use the following dynamics:

\begin{equation*}
    \begin{aligned}
        A_k(1) &= \begin{bsmallmatrix} 
        -0.2 & 1 \\ 
        -0.1 & 0.1 
       \end{bsmallmatrix} &B_k(1) &= \begin{bsmallmatrix} 
        1 & 0.5 \\ 
        2 & 0 
       \end{bsmallmatrix}&G_k(1) &= \begin{bsmallmatrix} 
        1 & 0 \\ 
        0 & 1 
       \end{bsmallmatrix}\\
A_k(2) &= \begin{bsmallmatrix} 
        0.2 & 0.1 \\ 
        -0.5 & 0.1 
       \end{bsmallmatrix} &
B_k(2) &= \begin{bsmallmatrix} 
        0 & 1 \\ 
        -1 & 2 
       \end{bsmallmatrix}&
G_k(2) &= \begin{bsmallmatrix} 
        0.5 & 0 \\ 
        0 & 0.5 
       \end{bsmallmatrix} \\
       \mathbb{P} &= \begin{bsmallmatrix} 
      0.8 & 0.2 \\ 
      0.9 & 0.1 
     \end{bsmallmatrix} & c_k &= \begin{bsmallmatrix}
         0.01\\0.01
     \end{bsmallmatrix} &
T & =6
    \end{aligned}
    \label{mjls_eg}
\end{equation*}

The initial and final conditions are set as $\rho_0=[0.3,0.7]$, ${\mu}_0=[25,40]^\top$, ${\Sigma}_0=6I$ and $\mu_f=[5,10]^\top$, $\Sigma_f=3I$. We set a half-space state chance constraint of $\mathbb{P}(a^\top x+b<0)\geq 0.95$, where $a=[0,-1]^\top$ and $b=-10$ and a control norm bound constraint $\mathbb{P}(\|u_k(i)\|\leq u_{\max})\geq0.95$ where $u_{\max}=8$. We consider a zero-mean gaussian distribution $\mathcal{N}(0,I_{n_w})$ for $w_k$. For sampling $x_0$, we also consider a normal distribution with mean $\mu_0$ and covariance $\Sigma_0$.

We use the CVXPY package\cite{diamond_cvxpy_2016} with MOSEK solver\cite{aps_mosek_nodate} to run \Cref{algorithm_1} on a 2023 Macbook Pro M1 Pro laptop with 16 GB of RAM. All initial weights are set at $\alpha_m=10^2$, with $\eta=1.5$ and $tol=10^{-6}$. It takes 3.43 seconds and $7$ iterations for \Cref{algorithm_1} to converge to the preset tolerance values. \Cref{fig:uc_traj,fig:cc_traj} plot 2500 Monte Carlo trajectories for this problem without and with state chance constraints, respectively. In both cases, the control norm chance constraint is also imposed. The initial covariance, terminal covariance, and the mean trajectory are shown in green, red, and blue, respectively, while Monte Carlo trajectories are shown in grey. The unshaded area indicates the feasible region under chance constraints. \Cref{fig:unorm_hist} shows the control norm history for the Monte Carlo simulations. The Monte Carlo simulations disobeyed the chance constraint in $0.01\%$ of total sample size, which meets the prescribed risk bound of $5\%$. \Cref{fig:cov_final} plots $x_T$ for all samples, overlayed with the terminal covariance constraint ellipsoid $\Sigma_f$, sample covariance and the predicted covariance $\Sigma_T$ from the optimizer. These numerical results demonstrate the performance of the controller to steer the covariance within the final target value in the given time horizon, while respecting the chance constraints. Moreover, the state trajectories in \Cref{statetraj} demonstrate the multi-modal nature of the problem because of the Markovian jump dynamics and mode-dependent controllers.  We also note the number of state distribution modes increasing over time, most apparent in \Cref{fig:uc_traj}, clearly demonstrating the exponential Gaussian mixture distribution of MJLS, which is hard to control using conventional stochastic optimal control approaches.

\begin{figure}[htbp!]
     \centering
     \begin{subfigure}{0.45\textwidth}
         \centering
         \includegraphics[width=\textwidth]{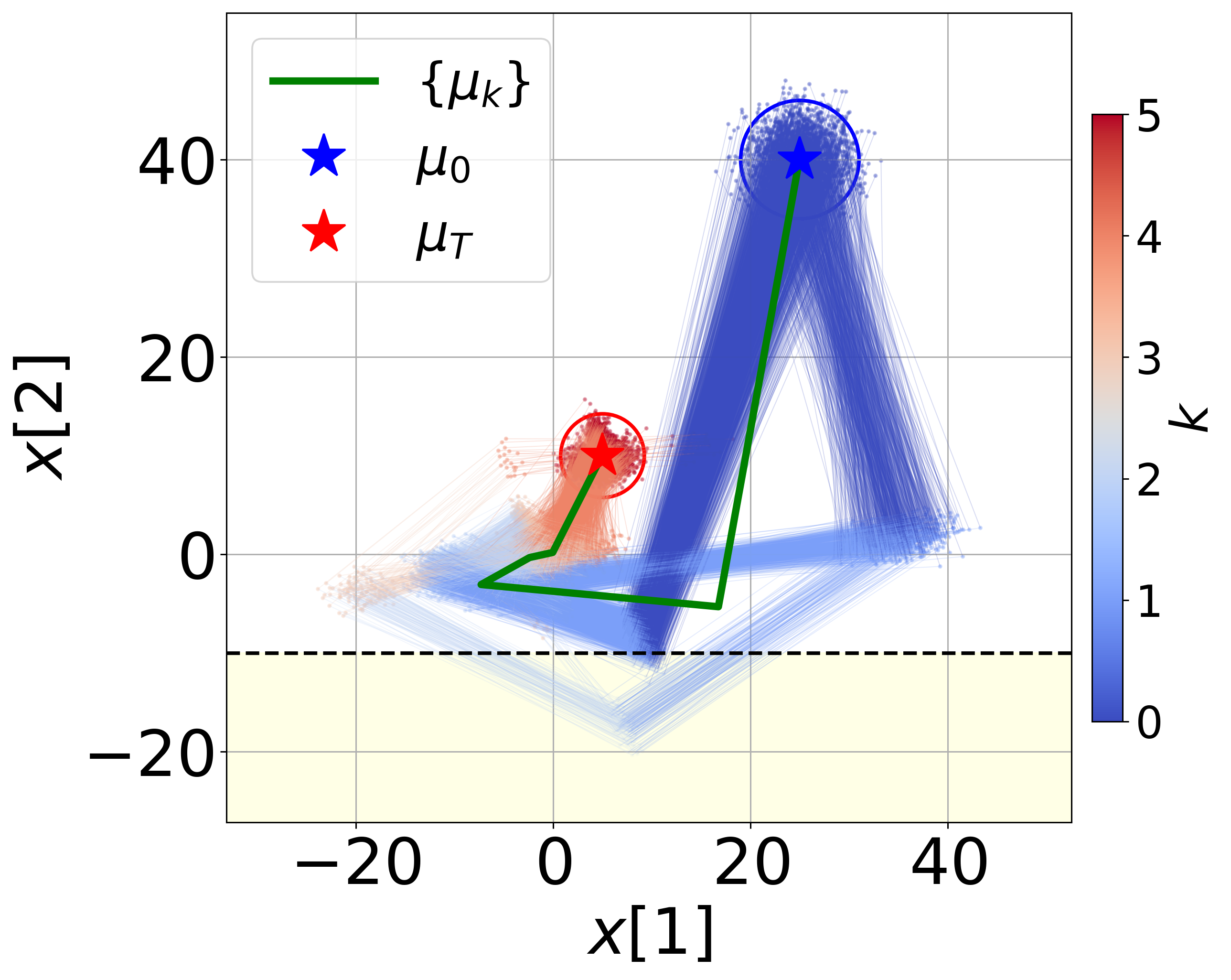}
         \caption{unconstrained case}
         \label{fig:uc_traj}
     \end{subfigure}
     \hfill
     \begin{subfigure}{0.45\textwidth}
         \centering
         \includegraphics[width=\textwidth]{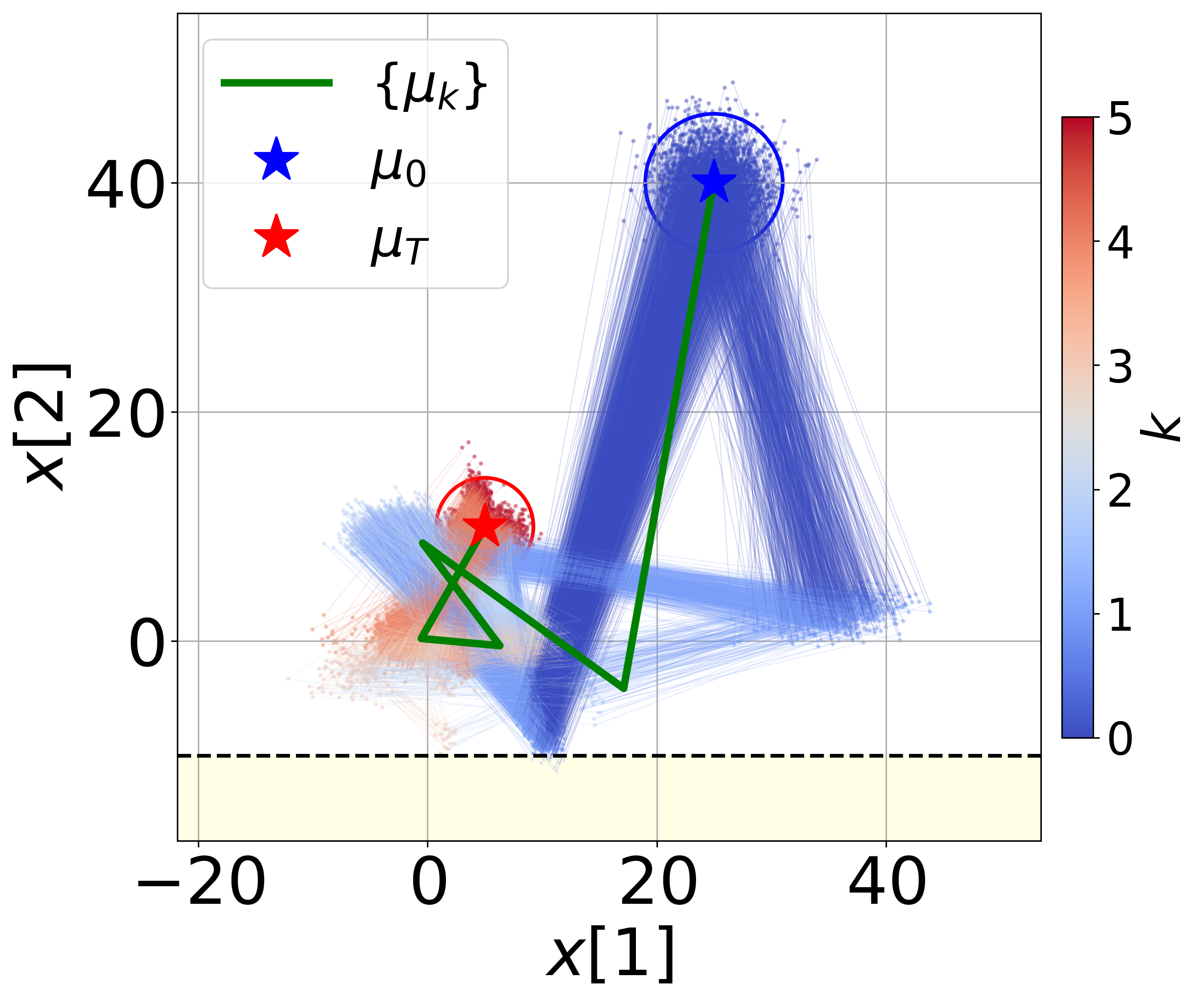}
         \caption{with chance constraints}
         \label{fig:cc_traj}
     \end{subfigure}
     \hfill
     \caption{State trajectories for Monte Carlo simulations with the state chance constraint $\mathbb{P}(x[2]\geq-10)\geq0.95$ (shaded region: infeasible.) Both cases use $P(\|u_k(i)\|\leq8)\leq0.95$.}
     \label{statetraj}
\end{figure}

\begin{figure}[htbp!]
    \centering
    \includegraphics[width= 0.7\linewidth]{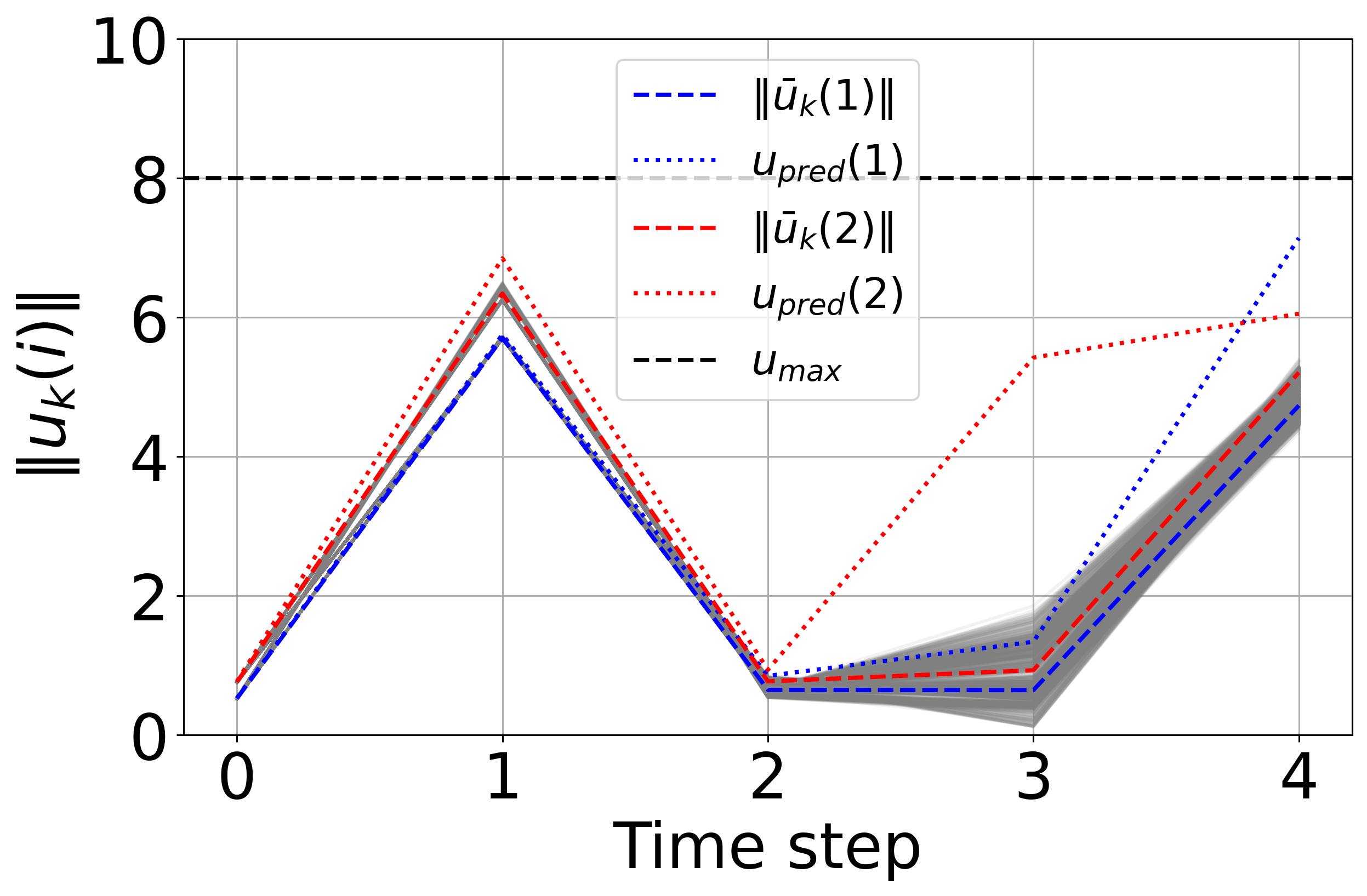}
    \caption{Control Norm History for the Monte Carlo simulations with state chance constraint $P(x_2\geq-10)\geq0.95$. $u_{pred}(i)$ corresponds to $\|\overline{u}_k(i)\|+\sqrt{\tfrac{n_u(i)}{\epsilon_u}\lambda_{\max}(\frac{Y_k(i)}{\rho_k(i)})}$}
    \label{fig:unorm_hist}
\end{figure}

\begin{figure}[htbp!]
    \centering
    \includegraphics[width= 0.7\linewidth]{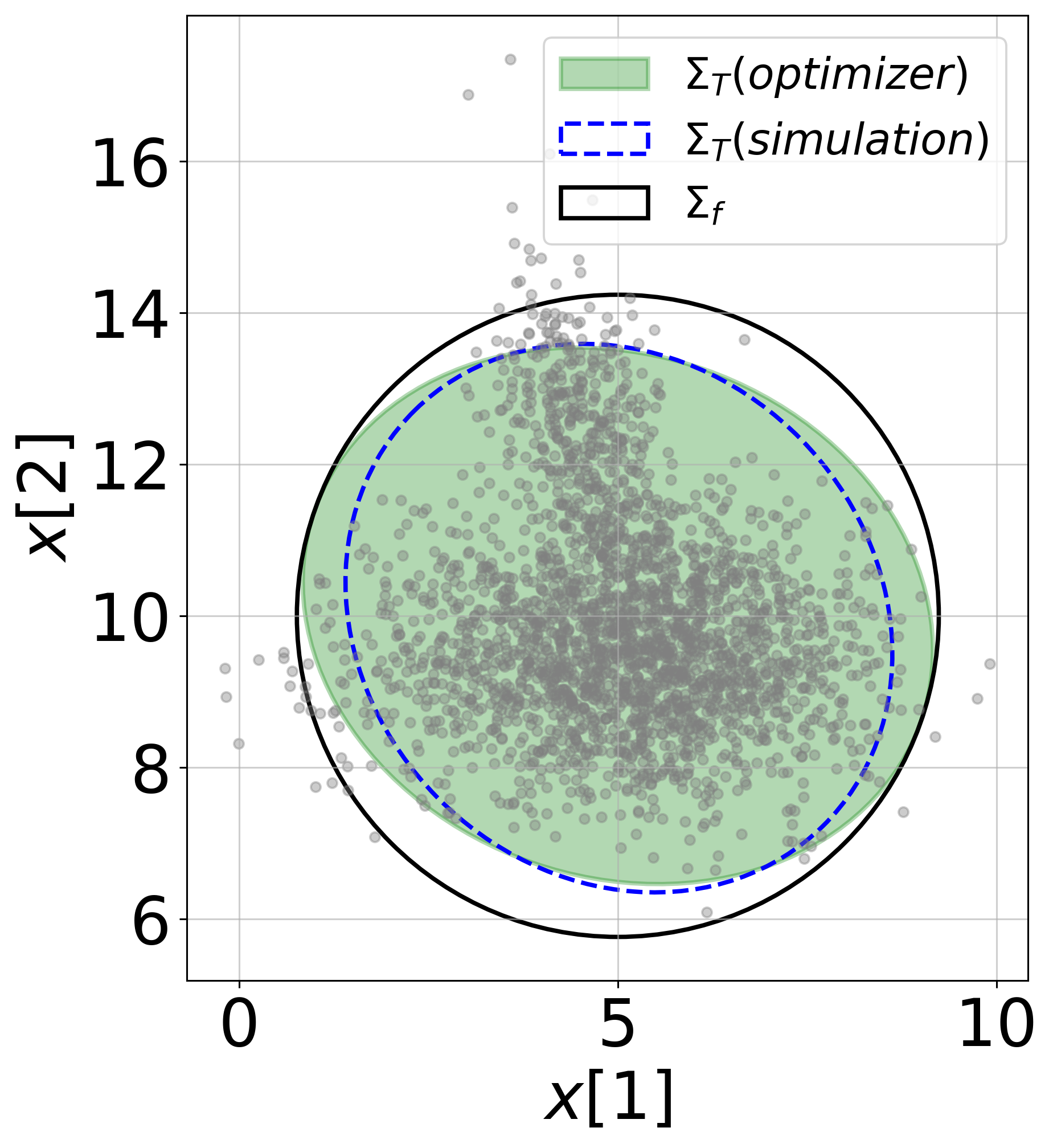}
    \caption{Sample points from Monte Carlo simulations at final time step, with sampled covariance covariance shown in blue, predicted covariance from optimizer in shaded green, and terminal covariance constraint shown in black. The covariance ellipses are scaled to contain $95\%$ of data of a gaussian distribution with same mean and covariance.}
    \label{fig:cov_final}
\end{figure}

For our second example, we consider a more practical double integrator system with three modes: the first mode denotes nominal operations, the second mode denotes nominal operations with wind-based disturbance, and the third mode denotes zero control authority with wind-based disturbance $a_\mathrm{w}$. The dynamics for the system are defined below: 

\begin{equation*}
    \begin{aligned}
        A_k(1) &= A_k(2) = A_k(3) = \begin{bsmallmatrix}
        1 & 0 & \Delta t & 0 \\
        0 & 1 & 0 & \Delta t \\
        0 & 0 & 1 & 0 \\
        0 & 0 & 0 & 1
        \end{bsmallmatrix}, \quad \Delta t = 1
    \end{aligned}
\end{equation*}
\begin{equation*}
    \begin{aligned}
        B_k(1) &= B_k(2) = \begin{bsmallmatrix}
        \frac{\Delta t^2}{2} & 0 \\
        0 & \frac{\Delta t^2}{2} \\
        \Delta t & 0 \\
        0 & \Delta t
        \end{bsmallmatrix}, \quad
        B_k(3) = \mathbf{0}_{4\times2}
    \end{aligned}
\end{equation*}
\begin{equation*}
    \begin{aligned}
        G_k(1) &= \mathbf{0}_{4\times4}, \quad G_k(2) = G_k(3) = 10^{-4}\,I_4
    \end{aligned}
\end{equation*}
\begin{equation*}
    \begin{aligned}
        c_k(1) &= \mathbf{0}_{4\times1}, \quad
        c_k(2) = c_k(3) = B_k(1)\, a_{\mathrm{w}}, \quad
        a_{\mathrm{w}} = \begin{bsmallmatrix} -0.15 \\ 0 \end{bsmallmatrix}
    \end{aligned}
\end{equation*}
\begin{equation*}
    \begin{aligned}
        \mathcal{P} &= \begin{bsmallmatrix}
        0.90 & 0.05 & 0.05 \\
        0.50 & 0.40 & 0.10 \\
        0.50 & 0.25 & 0.25
        \end{bsmallmatrix}, \quad T = 20
    \end{aligned}
\end{equation*}

The initial and final conditions are set as $\rho_0 = [0.5,\, 0.1,\, 0.4]$,
$\mu_0 = [2,\, {-3},\, 0,\, 0]^\top$, $\Sigma_0 = 10^{-4}\,I_4$ and
$\mu_f = [1.5,\, 3,\, 0,\, 0]^\top$, $\Sigma_f = 0.1\,I_4$.
We set a half-space state chance constraint $\mathbb{P}(a^\top x_k + b \leq 0) \geq 0.99$,
where $a = [-1,\, 0,\, 0,\, 0]^\top$ and $b = 0$,
and a control norm bound constraint $\mathbb{P}(\|u_k(i)\| \leq u_{\max}) \geq 0.95$
where $u_{\max} = 5$. We sample $x_0$ from a Gaussian variable with given mean and
covariance, and use a standard normal distribution for $w_k$. We note that both the wind in the second and third mode drives the system towards the unsafe zone.

Using the same solver, hardware set-up, and initial weights as before, \cref{algorithm_1} converges in 5 iterations under 79 seconds. Monte Carlo trajectories for this case are shown in \Cref{fig:db_cc_traj,fig:db_uc_traj} for the unconstrained and state-constrained case. The control chance constraint is imposed in both cases. Monte Carlo trajectories, control norm history, and terminal covariance are shown in \Cref{fig:db_uc_traj,fig:db_cc_traj,fig:db_unorm_hist,fig:db_cov_final}, following the same plotting conventions as the previous example.The violation rate for this example is $0.02\%$, which meets our prescribed risk bound of $5\%$. This example continues to demonstrate the controller's performance despite the Markovian jump dynamics, while also exhibiting the scalability with higher dimensions. We considered a much longer time horizon of $20$ time steps with $3$ modes, but \cref{algorithm_1} converged in less than a minute. Moreover, considering the wide class of practical dynamical systems which can be expressed as a double integrator system, this example also strongly demonstrates the practical applicability of our framework. It is also worth noting that the multi-modal nature is also exhibited in this example, which is most noticable in the unconstrained trajectory in \Cref{fig:db_uc_traj}. The control history plot in \cref{fig:db_unorm_hist} also shows that the predicted control effort norm for modes 1 and 2 differs noticeably, with the wind-induced disturbance mode resulting in a higher feedback component, as we would expect. 

\begin{figure}[htbp!]
     \centering
     \begin{subfigure}{0.45\textwidth}
         \centering
         \includegraphics[width=\textwidth]{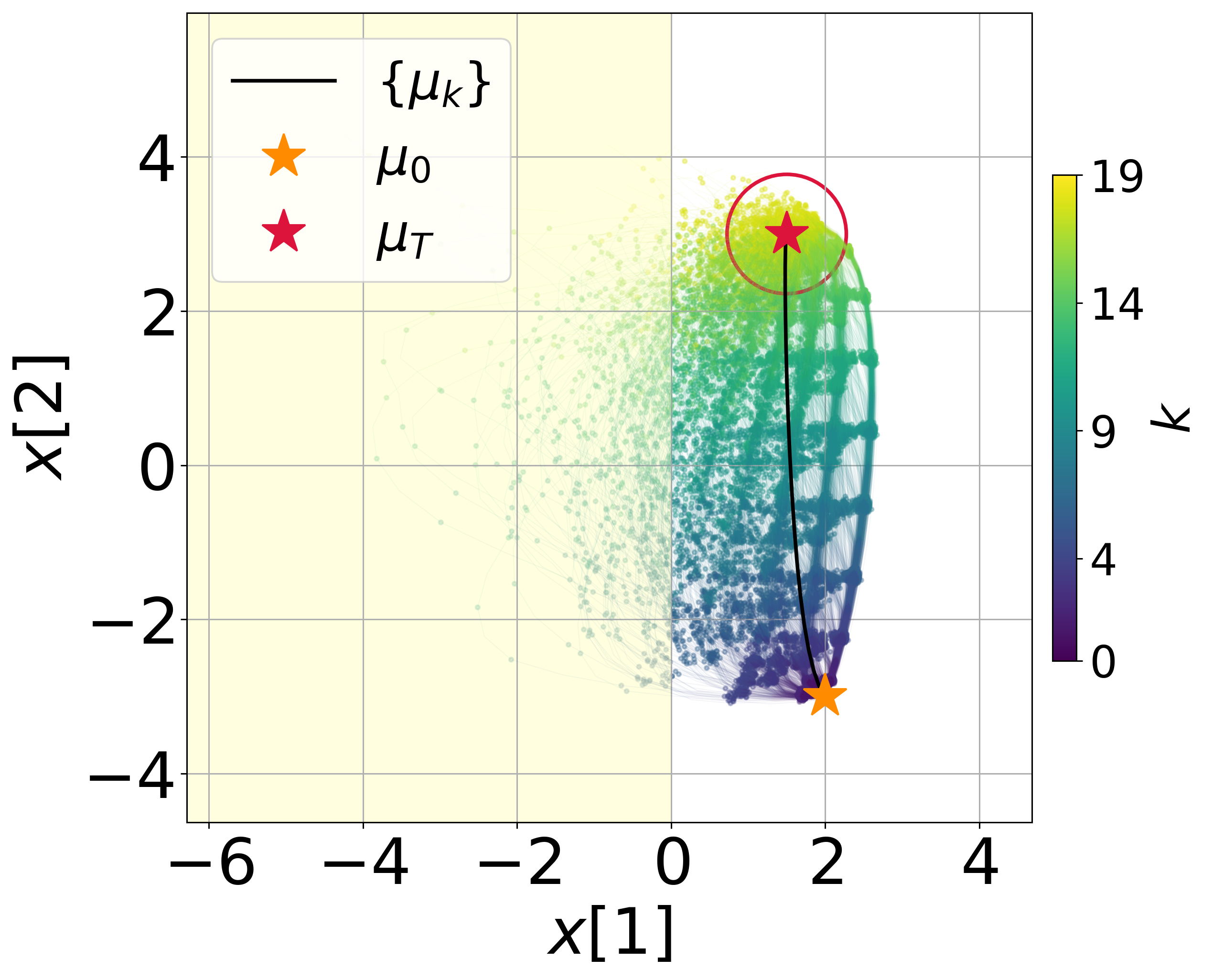}
         \caption{unconstrained case}
         \label{fig:db_uc_traj}
     \end{subfigure}
     \hfill
     \begin{subfigure}{0.45\textwidth}
         \centering
         \includegraphics[width=\textwidth]{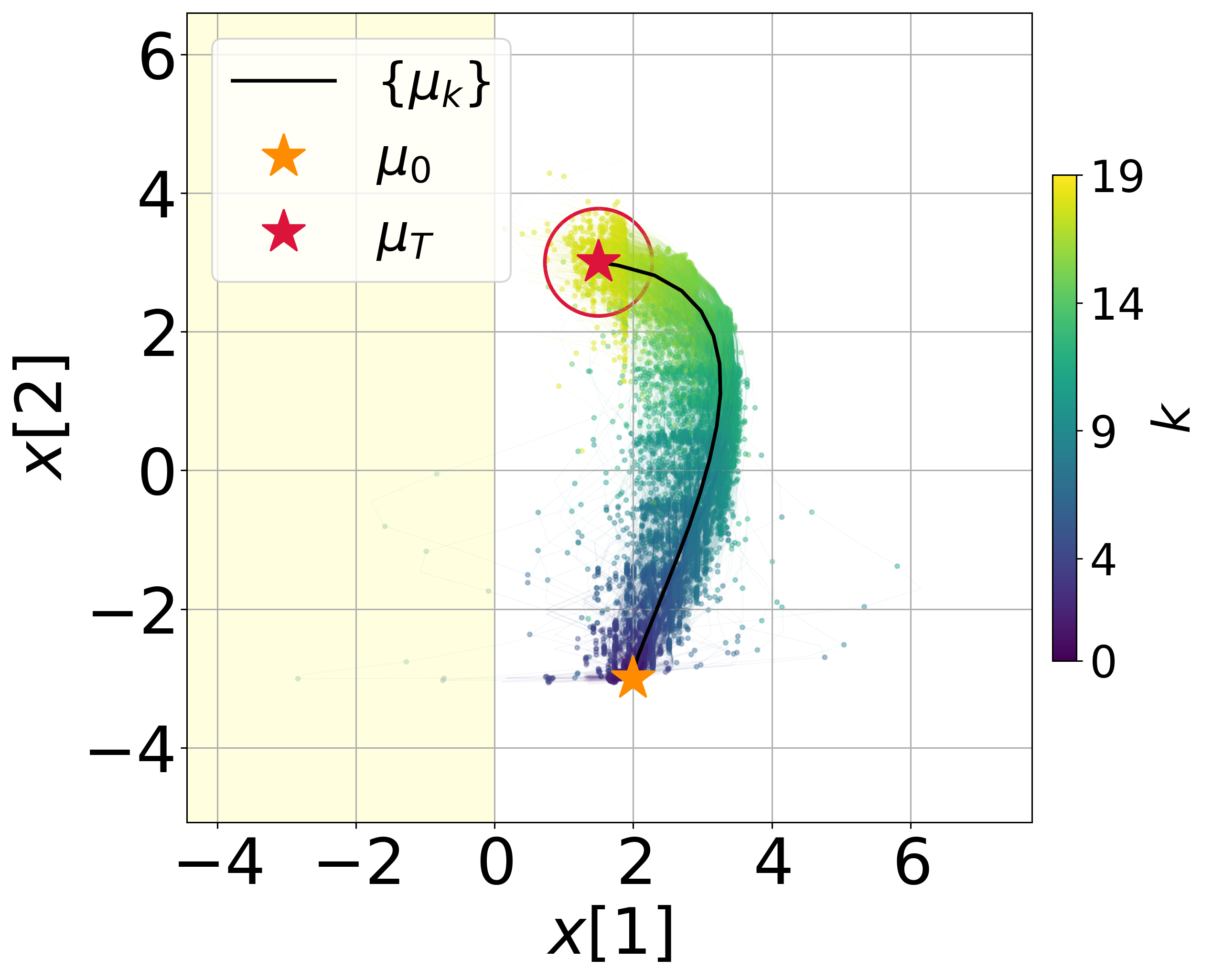}
         \caption{with chance constraints}
         \label{fig:db_cc_traj}
     \end{subfigure}
     \hfill
     \caption{State trajectories for Monte Carlo simulations with the state chance constraint $\mathbb{P}(x[1]\geq0)\geq0.95$ (shaded region: infeasible.) Both cases use $P(\|u_k(i)\|\leq5)\leq0.95$.}
     \label{statetraj}
\end{figure}

\begin{figure}[htbp!]
    \centering
    \includegraphics[width= 0.7\linewidth]{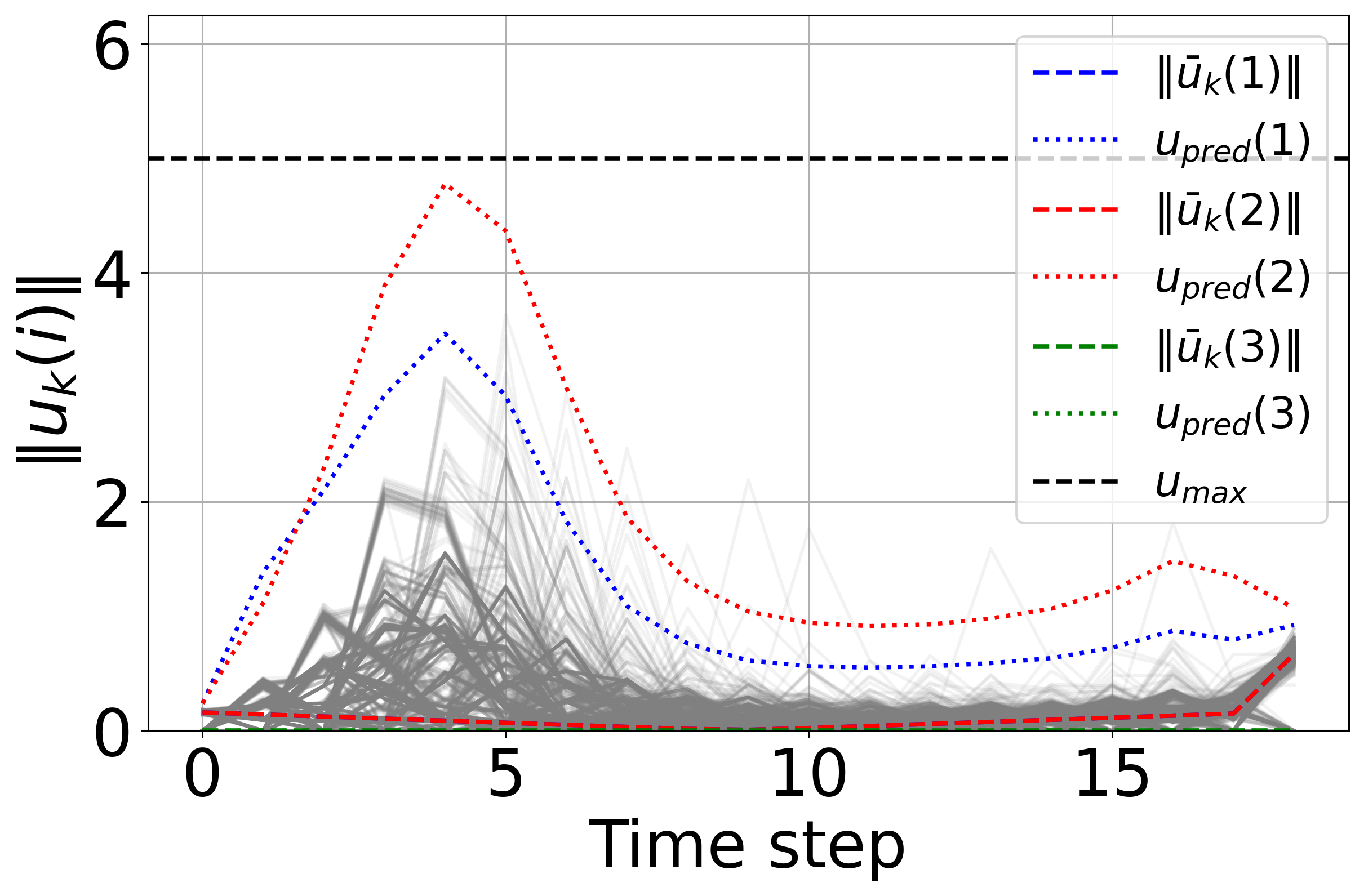}
    \caption{Control Norm History for the Monte Carlo simulations with state chance constraint $P(x[1]\geq0)\geq0.95$. $u_{pred}(i)$ corresponds to $\|\overline{u}_k(i)\|+\sqrt{\tfrac{n_u(i)}{\epsilon_u}\lambda_{\max}(\frac{Y_k(i)}{\rho_k(i)})}$}
    \label{fig:db_unorm_hist}
\end{figure}

\begin{figure}[htbp!]
    \centering
    \includegraphics[width= 0.6\linewidth]{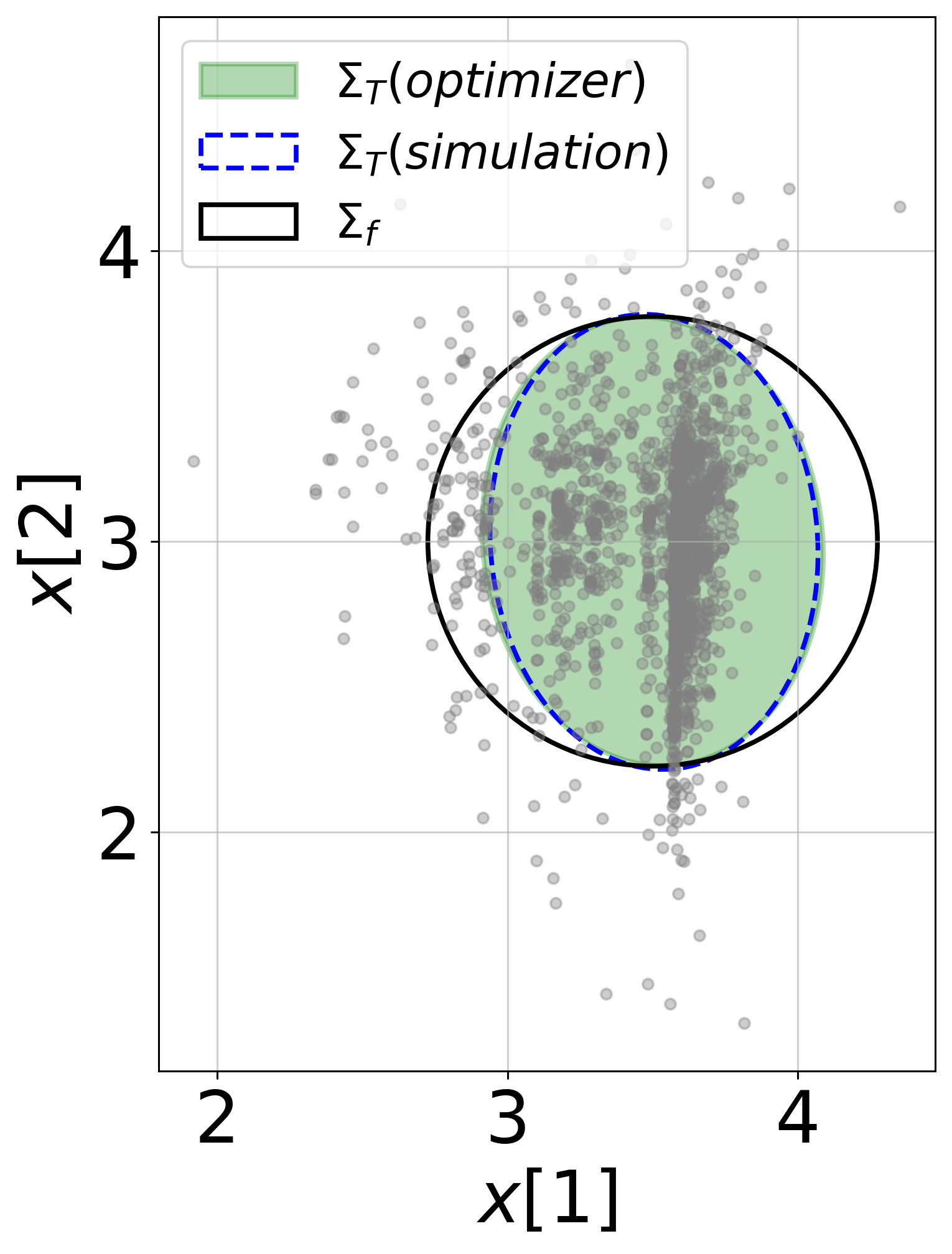}
    \caption{Sample points from Monte Carlo simulations at final time step, with sampled covariance covariance shown in blue, predicted covariance from optimizer in shaded green, and terminal covariance constraint shown in black. The covariance ellipses are scaled to contain $95\%$ of data of a gaussian distribution with same mean and covariance.}
    \label{fig:db_cov_final}
\end{figure}

\section{Conclusion}

In this paper, we consider the optimal covariance steering problem for MJLS under chance constraints. We demonstrate the coupled nature of the problem and decompose the problem into a mean-covariance problem, which can be solved in a two-step convex optimization framework. We propose lossless convex relaxations of the covariance subproblem in both unconstrained and chance-constrained cases and develop an iterative convex optimization framework to solve the chance-constrained covariance steering problem. Numerical simulations for the covariance steering framework are presented, which demonstrate the controller's performance in steering the covariance of a MJLS to within a given terminal covariance bound while obeying chance constraints with additive noise and bias.

\bibliographystyle{ieeetr}
\bibliography{references}
\appendices
\section{MJLS identities}

We show some of the identities which are used in deriving the covariance propagation equation for MJLS.
\begin{align}
    \mathbb{E}[u_k(i)|r_k=i]=\frac{\mathbb{E}[u_k(i)\mathds{1}_k=i]}{\mathbb{E}[\mathds{1}_{r_k=i}]}=\frac{\rho_k(i)\overline{u}_k(i)}{\rho_k(i)}=\overline{u}_k(i)
\end{align}
\begin{align}
    \mathbb{E}[(u_k(i)-\overline{u}_k(i))(u_k(i)-\overline{u}_k(i))^\top | r_k=i]\notag\\
    = \frac{\mathbb{E}[K_k(i)S_k(i)K^\top_k(i)\mathds{1}_{r_k=i}]}{\mathbb{E}[\mathds{1}_{r_k=i}]}=\frac{Y_k(i)}{\rho_k(i)}
\end{align}
\begin{align}
    \mathbb{E}[(x_k-\overline{x}_k(i))\mathds{1}_{r_k=i}]=0
\end{align}
\begin{align}\label{id0}
    &\mathbb{E}[x_kx_k^\top \mathds{1}_{r_k=i}]=\mathbb{E}[(x_k-\overline{x}_k(i)+\overline{x}_k(i))(x_k-\overline{x}_k(i)+\notag\\
    &\hspace{5mm}\overline{x}_k(i))^\top \mathds{1}_{r_k=i}]=S_k(i)+0+0+\rho_k(i)\overline{x}_k(i)\overline{x}^\top_k(i)\notag\\
    &=S_k(i)+\rho_k(i)\overline{x}_k(i)\overline{x}^\top_k(i)
\end{align}
\begin{align}\label{id1}
    &\mathbb{E}[x_ku_k^\top(i)\mathds{1}_{r_k=i}]=\mathbb{E}[x_k(\overline{u}_k(i) + K_k(i)\notag\\&\hspace{5cm}(x_k-\overline{x}_k(i)))^\top\mathds{1}_{r_k=i}]\notag\\  
    &=\mathbb{E}[((x_k-\overline{x}_k(i))+\overline{x}_k(i))(\overline{u}_k(i) +\notag\\&\hspace{3cm} K_k(i)(x_k-\overline{x}_k(i)))^\top\mathds{1}_{r_k=i}]\notag\\
      &=0+S_k(i)K^\top_k(i)+\rho_k(i)\overline{x}_k(i)\overline{u}^\top_k(i)+0\notag\\
      &=\rho_k(i)\overline{x}_k(i)\overline{u}^\top_k(i)+S_k(i)K^\top_k(i)
\end{align}
\begin{align}
    &\mathbb{E}[u_k(i)c_k^\top\mathds{1}_{r_k=i}]=\mathbb{E}[(\overline{u}_k(i) + K_k(i)(x_k-\overline{x}_k(i)))c_k^\top\mathds{1}_{r_k=i}]\notag\\
        &=\overline{u}_k(i)c_k^\top\mathbb{E}[\mathds{1}_{r_k=i}] + 0 = \overline{u}_k(i)c_k^\top\rho_k(i)
\end{align}
\begin{align}\label{id3}
    &\mathbb{E}[u_k(i)u^\top_k(i)\mathds{1}_{r_k=i}]=\mathbb{E}[(\overline{u}_k(i) + K_k(i)(x_k-\overline{x}_k(i)))(\overline{u}_k(i) +\notag\\& \hspace{4cm} K_k(i) (x_k-\overline{x}_k(i)))^\top\mathds{1}_{r_k=i}]\notag\\
        &=\mathbb{E}[\mathds{1}_{r_k=i}]\overline{u}_k(i)\overline{u}^\top_k(i)+\overline{u}_k(i)(\mathbb{E}[x^\top_k-\overline{x}^\top_k(i)\mathds{1}_{r_k=i}]
        K_k^\top(i)\notag\\& \hspace{3cm}+ K_k(i)\mathbb{E}[x_k-\overline{x}_k(i)\mathds{1}_{r_k=i}]\overline{u}_k^\top(i)\notag\\
        &\hspace{10mm}+K_k(i)\mathbb{E}[(x_k-\overline{x}_k(i))(x_k-\overline{x}_k(i))^\top\mathds{1}_{r_k=i}]K_k^\top(i)\notag\\  
        &=\rho_k(i)\overline{u}_k(i)\overline{u}^\top_k(i)+0+0+K_k(i)S_k(i)K^\top_k(i)\notag\\
        &=\rho_k(i)\overline{u}_k(i)\overline{u}^\top_k(i)+K_k(i)S_k(i)K_k^\top(i)
\end{align}

\end{document}